\documentclass{article}

\usepackage[applemac]{inputenc}		
\usepackage{authblk}

\usepackage{amsmath}
\usepackage{amsthm}
\usepackage{amsxtra}
\usepackage{pxfonts}
\usepackage{MnSymbol}
\usepackage{marvosym}

\usepackage{braket}
\usepackage{tikz} 

\usepackage{url}
\usepackage{hyperref}

\DeclareMathAlphabet{\mathsc}{OT1}{cmr}{m}{sc}

\newcommand{\lo}[1]{\raisebox{-0.1ex}{#1}\,}
\newcommand{\loo}[1]{\raisebox{-0.2ex}{#1}\,}
\newcommand{\Lo}[1]{\raisebox{-0.3ex}{#1}\,}
\newcommand{\Loo}[1]{\raisebox{-0.4ex}{#1}\,}

\newcommand{\LOO}[1]{\raisebox{-0.6ex}{#1}\,}

\newcommand{\C}{\mathbb C}
\newcommand{\N}{\mathbb N}

\newcommand{\eps}{\varepsilon}
\newcommand{\abs}[1]{\lvert #1 \rvert}
\newcommand{\norm}[1]{\lVert #1 \rVert}

\newcommand{\D}[1]{\mathrm{d}#1}

\newcommand{\e}{\mathrm{e}}

\newcommand{\mc}[1]{\mathcal{#1}}
\newcommand{\mf}[1]{\mathfrak{#1}}
\newcommand{\cat}[1]{\mathsc{#1}}

\DeclareMathOperator{\der}{der}

\DeclareMathOperator*{\colim}{colim}
\DeclareMathOperator{\ext}{ext}

\theoremstyle{definition}
\newtheorem{defn}{Definition}

\theoremstyle{plain}
\newtheorem{prop}[defn]{Proposition}
\newtheorem{thm}[defn]{Theorem}
\newtheorem{coro}[defn]{Corollary}

\theoremstyle{remark}
\newtheorem{ex}[defn]{Example}
\newtheorem{rk}[defn]{Remark}

\title{
On Complex Manifolds \\ and Observable Schemes
}
\author{Rodrigo Vargas Le-Bert\footnote{
\Letter {\tt rvargas@inst-mat.utalca.cl}. Supported by Fondecyt Postdoctoral Grant N\textordmasculine 3110045. 
} }
\affil{
Instituto de Matemática y Física, Universidad de Talca \\
Casilla 747, Talca, Chile
}
\date{March 26, 2012}

\begin{document}
\maketitle

\begin{abstract}
We work out the construction of a Stein manifold from a commutative Arens-Michael algebra, under assumptions that are mild enough for the process to be useful in practice. 
Then, we do the passage to arbitrary complex manifolds by proposing a suitable notion of \emph{scheme.} We do this in the abstract language of \emph{spectral functors,} in view of its potential usefulness in non-commutative geometry.
\end{abstract}

\tableofcontents

\newpage

\section{Introduction}

The duality between  spaces and  observable algebras can be counted among the most fruitful themes in mathematics and physics. It has different incarnations, depending on the kind of space one is interested in: algebraic varieties and polynomial rings, topological spaces and commutative C*-algebras, and quantum spaces and non-commutative operator algebras, just to name some of the best known examples.  Thus, each category of spaces has a corresponding category of observable algebras. 
This paper grew out from the need of using that correspondence in the problem of constructing  certain underlying geometric space associated to a Lie algebra representation. More precisely, what was needed was an explicit set of minimal, easily verifiable hypothesis on the representation ensuring that the resulting space was a complex manifold. It came as a surprise that, although much is known about the subject, no sufficiently simple characterization of the algebras of holomorphic functions on a complex manifold was available.

It is interesting to note how our approach stands in relation to that of algebraic geometry. It turns out that the construction mentioned above produces, given a Lie algebra representation, a variety which is naturally embedded in a Hilbert space. 
Now, as there exist discontinuous linear functions defined on any infinite dimensional Hilbert space, the corresponding ring of polynomial functions admits discontinuous characters,\footnote{Do not confuse this with the existence of discontinuous characters defined all over the completion of this ring, which is a Fréchet algebra. The statement that characters of a Fréchet algebra are automatically continuous is Michael's conjecture, hopefully proved in the unpublished article~\cite{m:sten98}.} in contrast with the finite dimensional case of algebraic geometry. Thus, the  topology of the ring must be taken into account, and accordingly we place ourselves within the framework of Arens-Michael algebras.

Regarding known results,
there are several characterizations of algebras of holomorphic functions on complex spaces, see~\cite{m:Gold90}, but as we said above we have been unable to find in the literature one that fits our needs. The problem is that, generally, it is asked for conditions which seem hard to check in practice. Let us mention, for instance, Brooks result~\cite{m:Broo79} on the local existence of analytic structure in the spectrum of a uniform Fréchet algebra $A$, which asks (among several other things) that the algebra generated by the germs of Gelfand transforms $\hat a$ at a point, for $a\in A$, be a quotient of the algebra of germs of holomorphic functions. We hope to have come up with hypothesis that are generally more easily verified.
%

Let us briefly describe the contents of the paper.
In the second section we review some basic facts about Arens-Michael algebras. 
In the third one, we give a set of geometric hypothesis on a commutative Arens-Michael algebra which we prove to be necessary and sufficient for its character space  to be a Stein manifold. 
We have attempted to formulate those hypothesis in such a way that commutativity can be dropped, with partial success. In the fourth and final section, we do the passage to the case of arbitrary complex manifolds. Although it is just a matter of ``gluing the affine pieces'', we wanted to have a theory that paralleled the elegance of schemes in algebraic geometry, and the crucial step is defining the right analogue of the Zariski topology. Keeping up  with our attempt to formulate things in such a way that commutativity is not essential, we define our topology using the novel notion of \emph{spectral functor,} meant to isolate the properties that a non-commutative, functorial spectrum must have in order for our theory to be applicable.

\section{Arens-Michael algebras}

All over this paper, $A$ will be an element of the category $\cat{amAlg}$ of unital Arens-Michael algebras, i.e.\ topological algebras whose topology is induced by a family
of  submultiplicative, unital seminorms. In this section we recall some known facts about them which will be used in the sequel.

\subsection{Seminorms}

Let $P=\mc P(A)$ be the set of all continuous, submultiplicative,  unital seminorms $p$ on $A$. 
It is partially ordered by the relation
\[
p\precsim q\quad \text{if, and only if,}\quad (\exists C>0)\ p\leq Cq,
\]
where we say that $p\leq q$ whenever $p(a)\leq q(a)$, for every $a\in A$ (this is also a partial order on $P$, but whenever we refer to $P$ as a partially ordered set, 
we will have $\precsim$ in mind).
Given $p,q\in  P$, the function
\[
(p\vee q)(a) = \max\{ p(a), q(a) \}
\]
is a continuous, submultiplicative,  unital seminorm. 
Thus, $P$ is a directed system. It is clearly \emph{saturated:} if $p\in P$ and $q\precsim p$, then $q\in P$. If $Q\subseteq P$ is pointwise bounded, then we define
\[
\bigl(\bigvee Q\bigr)(a) = \sup_{q\in Q} q(a).
\]
This seminorm does not necessarily belong to $P$ (we can only be sure that it is lower semi-continuous). However, if there exists $p\in P$ and $C>0$ such that $q\leq Cp$ for all $q\in Q$, then $\bigvee Q\precsim p$ and, in particular, $\bigvee Q\in P$.

\begin{defn}
Given $p,q\in P$, its minimum is
\[
p\wedge q = \bigvee \Set{r\in P | r\leq p,\ r\leq q }.
\]
We will follow the convention that $\bigvee\emptyset \equiv 0$,
which is not unital but for convenience will be admitted in $P$. Thus, $p\wedge q$ always belongs to $P$.
\end{defn} 

Given $p\in P$, we let $A_p$ be the Banach algebra
\[
A_p = \text{completion of } A/N_p\lo, \quad N_p = \Set{a\in A| p(a) = 0}.
\]
When $p\precsim q$,  there exists a Banach algebra morphism $\pi_{pq}: A_q\rightarrow A_p$\lo, and the construction is functorial. The morphism $\pi_{pq}$ has dense range and $\set{A_p | p\in P}$ is, whence, a dense directed system of Banach algebras.
Using the Arens-Michael representation, 
\[
A \cong \lim_{p\in P} A_p
 	= \Bigl\{ (a_p)\in\prod_{p\in P} A_p \Bigm\vert \pi_{pq}(a_q) = a_p \Bigr\}.
\] 
We will write $\pi_p:A\rightarrow A_p$ for the canonical morphisms of the projective limit and, by a harmless abuse of notation, identify $A_p^*$ with $\pi_p^*(A_p^*)\subseteq A^*$. We recall here the abstract Mittag-Leffler theorem: the range of $\pi_p$ is dense in $A_p$\lo, see \cite[Lemma 3.3.2]{m:Gold90}, for instance.

\subsection{The Gelfand spectrum}

In this subsection we merely collect a few facts regarding the spectrum of a commutative topological algebra, which plays a central role in its geometric study. As references, we mention \cite{m:Mall86, m:Frag05}.

Let $M=\mc M(A)$ be the Gelfand spectrum of  $A$, i.e.\ the set of continuous characters
\[
M = \Set{ \chi\in A^* | (\forall a,b\in A)\ \chi(ab) = \chi(a)\chi(b) }
\]
equipped with (the trace of) the $\sigma(A^*, A)$-weak topology. There is a bijective correspondence between continuous characters $\chi\in M$ and closed maximal ideals $\mf m$ of $A$, given by $\mf m = \ker \chi$.
Given $p\in P$, we will write $M_p\subseteq M$ for the subset of $p$-continuous characters. Note that $M=\bigcup_{p\in P} M_p$\loo. In particular, if $A$ is commutative, then $M\neq\emptyset$.\footnote{The spectrum of an arbitrary topological algebra can be empty.} 

Consider some $a\in A$. Its Gelfand transform is the function $\hat a: M\rightarrow \C$ defined by
\(
\hat a(\chi) = \chi(a).
\)
Thus, the Gelfand transform $\mc G:a\in A\mapsto \hat a\in C(M)$ is a unital algebra morphism, which is injective if, and only if, $A$ is semisimple (and then it is isomorphic to a subalgebra of $C(M)$). If $A$ is Fréchet (i.e.\ its topology can be defined by a countable family of seminorms), then the Gelfand transform is automatically continuous.

\section{From algebras to Stein manifolds}

\subsection{Geometric hypothesis} \label{geometric hypothesis}

In this subsection we introduce and discuss the hypothesis on $A$ which allow us to represent it as an algebra of analytic functions on a Stein manifold. They are six: commutativity, semisimplicity, local exponentiability, non-singularity, finite dimensionality, and a  technical hypothesis asking that given a closed maximal ideal $\mf m\subseteq A$, there exists a basis of the cotangent space $\mf m/\mf m^2$ which gives, upon taking representatives in $\mf m$ of its elements, a set of topological generators for~$A$. 

This is no place for much commentary on the hypothesis of commutativity,
except perhaps for a remark on the problems arising upon trying to drop it. In the framework we have chosen to work with, the difficulty lies in reformulating the notion of regularity and the technical hypothesis mentioned above, as we will soon see. 

Semisimplicity is just to ensure that the Gelfand transform is injective. Thus, it is not essential, and one could drop it if a notion of ``observable'' more general than that of functions on a manifold is welcome.\footnote{And, correspondingly, a notion of ``geometry'' more general than that of a manifold.}

The local exponentiability hypothesis concerns derivations.
Recall that a \emph{derivation} of $A$ is a linear map $\delta:A\rightarrow A$ such that
\[
\delta(ab) = \delta(a)b + a\delta(b),\quad a,b\in A.
\]
Given a character $\chi\in M$, a \emph{derivation over $\chi$} is a linear map $\delta_\chi:A\rightarrow \C$ such that
\[
\delta_\chi(ab) = \delta_\chi(a)\chi(b) + \chi(a)\delta_\chi(b),\quad a,b\in A.
\]
The $A$-modules of continuous derivations and continuous derivations over $\chi$ will be denoted  $\der A$ and $\der_\chi A$, respectively. 
Local exponentiability of derivations is what determines the fact that we are dealing with a complex analytic geometric structure. 

\begin{defn}
We say that $A$ is \emph{locally exponentiable} if given $\delta\in\der A$ and $\chi_0\in M$, there exists a neighborhood $U\ni\chi_0$ such that the series
\[
z\mapsto \sum_{k\in\N} \frac{z^k}{k!}\chi\circ\delta^k(a)
\]
converges, for every $\chi\in U$, $a\in A$ and $\abs z< R$, where $R=R(\delta, U)>0$.
\end{defn}

We also need a hypothesis ensuring that $M$ has no singularities. A fact that characterizes singular points is that the dimension of their tangent space is larger than it should. In particular, at a singular point $\chi\in M$, there exist tangent vectors that cannot be extended as vector fields on a neighborhood of~$\chi$. In order to adopt this as a definition,
we use the canonical map $\der A\rightarrow \der_\chi A$ given by $\delta\mapsto \chi\circ\delta$.
\begin{defn}
We say that $A$ is \emph{non-singular} 
if the canonical map $\der A\rightarrow \der_\chi A$ is surjective, for all $\chi\in M$.
\end{defn}
\begin{rk}
This definition does not pass directly to the non-commutative case, because there the set of multiplicative characters is not a good replacement of the spectrum. One way to get over this difficulty, by the usual method of avoiding any mention of ``points'', would be to say that $A$ is non-singular if $\der A$ is projective and finitely generated. However, it is not known whether this is equivalent to the above in the commutative case (a problem 
that seems related to the Nakai and Zariski-Lipman conjectures).
\end{rk}

Let $\chi\in M$ and $\mf m = \ker\chi$. In order to ensure that $M$ is finite dimensional, we  suppose that $\mf m/\mf m^2$ is finite dimensional. 
Finally, we make the following technical hypothesis: 
there exists a set of topological generators $\{a_i\}\subseteq A$ such that $\bigl\{ [a_i-\chi(a_i)] \bigr\}$ is a basis of $\mf m/\mf m^2$.
An important question that we leave unanswered is to what extent this hypothesis can be weakened,
and how it can be reformulated in a manner suitable for generalization to the non-commutative case.

\begin{rk}
Determining the spectrum of topologically finitely generated Fréchet algebras is equivalent to determining the polynomially convex hull of compact subsets of~$\C^n$, see \cite[Chapter 5]{m:Gold90}.
\end{rk}

\subsection{The reconstruction theorem}

The following result is, basically, an adaptation of Connes' openness lemma~\cite{m:Conn08} to our context.

\begin{thm} \label{reconstruction}
Under the geometric hypothesis on $A$ above, every $\chi_0\in M$ has a  neighborhood $U\ni \chi_0$ which is homeomorphic to $\C^n$, and the change of coordinate maps are holomorphic. Thus, $M$ is a Stein manifold.
\end{thm}
\begin{proof}
Write $\mf m=\ker\chi_0$\lo,
let $\{a_i\}\subseteq A$ be a topologically generating set such that $\bigl\{ [a_i-\chi_0(a_i)] \bigr\}$ provides a basis of $\mf m/\mf m^2$, and let $\{\delta_i\}$ be a dual basis of $\der_\chi A \cong (\mf m/\mf m^2)^\dagger$, meaning that one has
\[
\chi_0(\delta_i(a_j)) = \delta_{ij}\loo,\quad i,j\in\{1,\dots,n\}.
\] 
By regularity,  each $\delta_i$ is the image under $\der A\rightarrow \der_{\chi_0} A$ of a derivation $A\rightarrow A$ which we will still denote by $\delta_i$\loo. By local exponentiability, the power series 
\[
\varphi(z) = \sum_{\alpha\in\N^n} \frac{z^\alpha}{\alpha!} \chi_0\circ \delta_1^{\alpha_1}\circ\cdots\circ \delta_n^{\alpha_n} 
\]
converges pointwise on a polydisc $D\subseteq \C^n$, thus defining a map $\varphi:D\rightarrow M$. Now, consider the map $\psi:M\rightarrow \C^n$ given by $\chi\mapsto (\chi(a_i))$. By construction,
\[
(\psi\circ\varphi)'(0) = \text{Id},
\]
and thus $\psi\circ\varphi$ restricts to a holomorphic bijection between the open sets
\[
0\in V\subseteq \C^n \rightarrow (\chi_0(a_i)) \in W\subseteq \C^n.
\]
We let $U = \psi^{-1}(W)\subseteq M$. In order to conclude that the family of maps $\psi:U\rightarrow \C^n$ thus obtained provides a holomorphic atlas on $M$, it suffices to show that $\psi$ is injective, and this is because different characters cannot  take the same values on all of the generating set $\{a_i\}$.
\end{proof}


\begin{rk}
One nice property of the geometric hypothesis is that analiticity is neatly isolated: if one was interested in smooth manifolds, it would suffice to ask for \emph{local integrability,} as opposed to the stronger notion of local exponentiability, of the differential equation
\(
\frac{\D}{\D t} \chi_t = \delta^*\chi_t\loo.
\)
\end{rk}

\section{Spectra and schemes}

After Theorem~ \ref{reconstruction}, the next step is passing from the affine case to schemes.\footnote{There exists the notion of \emph{Gelfand sheaf} associated to a topological algebra~\cite{m:Mall98}, but it cannot be considered a scheme for it lacks a localization method.} In order to do so, we propose an analogue of the Zariski topology in the context of Arens-Michael algebras. Our definition depends essentially only on the lattice of continuous seminorms of $A$, making it potentially useful in non-commutative settings.
We have attempted to leave that possibility open, but that does not come for free, as we explain in the next paragraph.

The geometric approach to an algebraic object requires a theory of the corresponding dual object, or structure space. In the case of a commutative C*-algebra, the dual object is given by the Gelfand spectrum, but once one enters the non-commutative world there is more than one candidate to fulfill that role, see for instance~\cite{m:Shul82, m:Fuji98}. 
Since it is not our purpose  
to try and develop the ``right'' duality theory for non-commutative Arens-Michael algebras, we have decided to work with a minimal structure that fits our needs, hence the introduction of \emph{convex structures} on topological spaces and \emph{spectral functors.} We believe that these notions serve at least to focus in what matters for what we have to say: namely, that we have a proposal regarding the topology with which the dual object to an Arens-Michael algebra should be equipped. We do not claim more than that, and the only example of spectral functor that we will provide is the pure state space.

\subsection{Spectral functors}

\begin{defn} 
A \emph{convex structure} on a set $M$ is a subcategory 
\(
\sigma\subseteq\cat{Subset}(M) 
\)
which:
\begin{enumerate}
	\item Covers $M$,
	\item Is directed, 
	\item Is closed under finite intersections. 
\end{enumerate}
Elements of $\sigma$ will be called \emph{convex sets.}
A \emph{coarse map}  between spaces $(M,\sigma)$ and $(N,\tau)$ equipped with convex structures is a functor $\sigma\rightarrow \tau$\lo.
The category of convex-structured spaces and coarse maps will be denoted by $\cat{Conv}$. 
\end{defn}
\begin{defn} \label{spectral functor}
Let $\cat{sAlg}$ be some subcategory of $\cat{amAlg}$. A functor 
\[
A\in\cat{sAlg}\mapsto \bigl(\mc M(A),\sigma(A)\bigr)\in\cat{Conv}
\]
is said to be \emph{spectral} if, for each $A\in\cat{sAlg}$, there exists a morphism of directed sets $p\in \mc P(A)\mapsto \mc M_p(A)\in \sigma(A)$ such that, for all $p,q\in \mc P(A)$, one has:
\begin{enumerate}
	\item $\mc M_p(A)\subseteq \mc M_q(A)$ if, and only if, $p\precsim q$.
	\item 
\(
\mc M_{p\wedge q}(A) = \mc M_p(A)\cap \mc M_q(A).
\)
\end{enumerate}
\end{defn}

The most important example of a spectral functor is given by the Gelfand spectrum, which is actually more than that: it is a functor between the categories of commutative Arens-Michael algebras and topological spaces. We consider the weaker notion above in order to accommodate for the simplest-minded generalization to the non-commutative setting, namely pure state spaces.\footnote{Note that, in the terminology of representation theory, \emph{characters} are what we call \emph{states,} the word ``character'' being reserved here for \emph{multiplicative} states.} We start our exposition by recalling the theory in the context of Banach algebras. Thus, during the next few paragraphs, $A$ will be a Banach algebra.

\begin{defn}
Let $A$ be a unital Banach algebra. 
An $f\in A^*$ is called a \emph{state} if
\[
\norm f := \sup_{\norm a\leq 1}\abs{f(a)} = f(1).
\]
States will usually be normalized by $f(1) = 1$. We write $\mc S(A)$ for the set of all normalized states.
\end{defn}
\begin{rk}
Being a state depends on the particular norm used to define the topology of~$A$. 
Linear functionals $A\rightarrow \C$ which are states for any submultiplicative, unital norm equivalent to $\norm\cdot$ are called \emph{spectral states,} see \cite{m:Pann90, m:BonsDunc71} for more.
They are characterized by the following equivalent properties:
\begin{enumerate}
\item For every $a\in A$, $f(a)$ belongs to the convex envelope of the spectrum
\[
\sigma(a) = \set{\lambda\in \C| a-\lambda \text{ is not invertible}}.
\]
\item For every $a\in A$, $\abs{f(a)}$ is bounded by the spectral radius
\[
\rho(a) = \sup_{\lambda \in\sigma(a)} \abs\lambda = \lim p(a^n)^{1/n}.
\]
\end{enumerate}
Spectral states are  appealing, but in non-commutative cases there are too few of them: for $A=\mathrm{M}_n(\C)$, the only spectral state is the trace. In particular, spectral states do not separate points.
\end{rk}

We now define \emph{pure states.} We need the following simple fact.
\begin{prop} \label{S(A) is compact}
$\mc S(A)$ is compact. 
\end{prop}
\begin{proof}
It suffices to check that $\mc S(A)$ is $\sigma(A^*, A)$-closed. So, let $\{f_i\}_{i\in I}$ be a net which is $\sigma(A^*, A)$-convergent to $f\in A^*$.  Given $a\in A$ with $\norm a\leq 1$,
\[
f(a) = \lim f_i(a) \leq \liminf \norm{f_i} = \liminf f_i(1) = 1,
\]
and thus $\norm f\leq 1$. We conclude by noting that $f(1) = \lim f_i(1) = 1$.
\end{proof}

\begin{defn}
Note that $\mc S(A)$ is convex. We define the \emph{pure state space} $\mc{PS}(A)$ 
to be the set of extreme points $\ext \mc S(A)$, which is non-empty because of Proposition~\ref{S(A) is compact}.
\end{defn}

\begin{rk}
In the commutative, C*-algebraic case, pure states are characters and the character space, being a closed subset of $\mc S(A)$, is compact. This is not so anymore in general, see~\cite{m:Saka59, m:Arch89},  and in certain cases it is necessary to take the closure of the set of pure states~\cite{m:Glim60}.
\end{rk}

We finish our review of Banach algebra state theory with the following result due to Moore~\cite{m:Moor71, m:Sinc71}.
\begin{prop} \label{Moore}
The  states $\mc{S}(A)$ span the dual $A^*$. In particular, pure states separate points of a Banach algebra. 
\end{prop}

We are ready to let $A$ be again a unital Arens-Michael algebra.
Recall that $A^* = \bigcup A_p^*$ (to be precise, $A^* = \colim A_p^*$ equipped with the final locally convex topology~\cite{m:Mall86}, and we identify $A_p^*$ with its image under the natural map $A_p^*\rightarrow \colim A_p^*$). 
\begin{defn}
A \emph{state} of an Arens-Michael algebra $A$ is a linear functional $f\in A^*$ which is a state of $A_p$ for some $p\in\mc P(A)$.
The set of normalized states of $A$ will be written $\mc S(A) = \bigcup \mc S(A_p)$,
Finally, we define the \emph{pure state space}  to be the union $\mc {PS}(A) = \bigcup_{p\in P} \mc{PS}(A_p)$. 
\end{defn}

\begin{rk} \label{S_q subseteq S_p}
One might wonder what happens if we take the intersection rather than the union in $\mc S(A)=\bigcup \mc S(A_p)$. Now, observe that
if $f\in \mc S(A_q)$ and $p\geq q$, then 
\[
\sup_{p(a)\leq 1} \abs{f(a)} \leq \sup_{q(a)\leq 1} \abs{f(a)} = f(1),
\] 
where the supremums are taken over $a\in A$. Using the Mittag-Leffler theorem, we conclude that $f\in \mc S(A_p)$. 
Thus, taking the intersection leads to the union of all \emph{spectral states} of the $A_p$'s.
\end{rk}
\begin{rk}
Our definition of $\mc{PS}(A)$ seems somewhat arbitrary, inasmuch as pure states should be extreme points of $\mc S(A)$. However, pursuing the Choquet theory of Arens-Michael algebra state spaces would take us too far afield.
\end{rk}
\begin{rk}
Pure states separate points. Indeed,
by Proposition~\ref{Moore}, $\ext \mc S(A_p)$ separates points of $A_p$\lo. Since for every $a\neq 0$ there exists $p\in P$ such that $p(a)\neq 0$, the claim follows.
\end{rk}


The following result, needed to prove that pure states define a spectral functor, is due to Bohnenblust and Karlin~\cite{m:BohnKarl55}. 
\begin{prop} \label{BohnKarl}
Each seminorm $p\in P$ is equivalent to
\[
\norm a_{\mc S(A_p)} = \sup_{f\in \mc S(A_p)} \abs{f(a)},\quad a\in A.
\]
More precisely, one has that
\(
\norm a_{\mc S(A_p)}\leq p(a) \leq \e\norm a_{\mc S(A_p)}\Loo.
\)
\end{prop}
\begin{coro} \label{p equiv norm M_p}
Let $\mc M_p(A) = A_p^*\cap\mc{PS}(A)$. Each seminorm $p\in\mc P(A)$ is equivalent to
\[
\norm{a}_{\mc M_p(A)} =  \sup_{f\in\mc M_p(A)} \abs{f(a)}.
\]
\end{coro}
\begin{proof}
Since $\mc M_p(A)\subseteq \set{\norm f_{A_p^*}\leq 1}$, we have $\norm{\ }_{\mc M_p(A)}\leq p$.
On the other hand,
 $\mc S(A_p)$ is contained in the closure of the convex envelope of $\mc M_p(A)$, and thus
Proposition \ref{BohnKarl} implies 
 that $p\precsim\norm{\ }_{\mc M_p(A)}$\Loo. 
\end{proof}

\begin{thm} 
Let $\cat{sAlg}$ be the subcategory of those $A\in\cat{amAlg}$ such that 
\[
\bigl( \forall p,q\in \mc P(A) \bigr)\  A_p^*\cap A_q^* = A_{p\wedge q}^*\Loo.
\]
Then, 
\[
\mc M(A)=\mc{PS}(A),\quad \mc M_p(A) = A_p^*\cap\mc{PS}(A)  
\]
define a spectral functor $\cat{sAlg}\rightarrow \cat{Conv}$.
\end{thm}

\begin{proof}
Clearly, $\sigma(A) = \set{ \mc M_p(A) | p\in\mc P(A) }$ covers $\mc M(A)$, is directed, and is closed under finite intersections if $A\in\cat{sAlg}$; thus, it defines a convex structure on $\mc M(A)$, and we get that $\mc M:\cat{sAlg}\rightarrow \cat{Conv}$. Let us check functoriality: we must associate, to each morphism $\varphi:B\rightarrow A$, a coarse map $\mc M(A)\rightarrow \mc M(B)$, i.e.\ a functor $\sigma(A)\rightarrow \sigma (B)$. Our candidate is
\[
\mc M_p(A) \in \sigma (A) \mapsto \mc M_{p\circ\varphi}(B) \in \sigma(B),
\]
which is well-defined and functorial because, thanks to Corollary~\ref{p equiv norm M_p}, we have that 
\[
(\forall A\in \cat{amAlg})\bigl(\forall p,q\in \mc P(A)\bigr)\ p\precsim q \Leftrightarrow \mc M_p(A)\subseteq \mc M_q(A).
\]
This is also the first condition of spectrality, and the second one follows immediately form the fact that we  consider only algebras $A\in\cat{sAlg}$.
\end{proof}

\subsection{Fullness and $A$-convexity}

Fix once and for all a spectral functor $\mc M$. When there is no more than one algebra $A$ around, we will write $M=\mc M(A)$ and $M_p = \mc M_p(A)$. 

\begin{prop}
For any $m\in M$, there exists a unique  seminorm $p_m\in P$ which is minimal (with respect to $\leq$) for the property $m\in M_p$\loo. 
\end{prop}
\begin{proof}
In order to use Zorn's lemma we only have to check, given a totally ordered subset $Q\subseteq P$, that  $\inf_{q\in Q} q(a)$  is a submultiplicative seminorm.
Now, subadditivity will follow from the fact that
\[
\inf_{q\in Q} \bigl\{ q(a)+q(b) \bigr\} = \inf_{q\in Q} q(a) + \inf_{q\in Q} q(b).
\]
The left hand side is clearly greater than or equal to the right hand side. In order to prove equality, let $\eps>0$ and take seminorms $q_1, q_2\in Q$ such that $q_1(a) \leq \inf_{q\in Q} q(a)+\eps$ and $q_2(b) \leq \inf_{q\in Q} q(b) + \eps$. Since $Q$ is totally ordered, we can suppose without loss of generality that $q_1\leq q_2$\lo, and then
\[
\inf_{q\in Q} \bigl\{ q(a)+q(b) \bigr\} \leq q_1(a)+q_1(b) \leq \inf_{q\in Q} q(a) + \inf_{q\in Q} q(b) + 2\eps.
\]
Since $\eps>0$ is arbitrary, equality follows. This proves subadditivity, and submultiplicativity is done similarily. Uniqueness follows  from 
the fact that $P$ is closed under $\wedge$.
\end{proof}

\begin{defn}
Given $m\in M$, write $O(m) =  M_{p_m}$\Loo.
We say that a subset $N\subseteq M$ is \emph{full} if 
\[
(\forall m\in N)\ O(m) \subseteq N.
\]
The full subset of $M$ generated by $N$ will be written
\(
O(N) = \bigcup_{m\in N} O(m).
\)
\end{defn}
\begin{defn}
Given  $N\subseteq M$, we define
\[
\mc P(N) = \set {p\in P | M_p\subseteq N},
\]
(not to be confused with $\mc P(A)$ when $A$ is an algebra)
and we say that $N$ is $A$-convex if $\mc P(N)$ is directed. Note that this is the case if, and only if, given $p_1,p_2\in\mc P(N)$ one has that $p_1\vee p_2\in\mc P(N)$.
\end{defn}
\begin{rk}
If $N$ is full, then $N\subseteq \bigcup_{p\in\mc P(N)} M_p$\lo. If it is full and $A$-convex, then $N = \bigcup_{p\in\mc P(N)} M_p$\lo.
\end{rk}

\begin{ex}
$M$ is trivially full and $A$-convex. 
\end{ex}

\begin{prop}
The sets $M_p$ are full and $A$-convex.
\end{prop}
\begin{proof}
Let $m\in M_p$\lo. Since $p_m \leq p$ (and, in particular $p_m\precsim p$), we have that $O(m)\subseteq M_p$\lo, showing that $M_p$ is full. For $A$-convexity, let $\{q_i\}_{i\in I}\subseteq\mc P(M_p)$, i.e.\ $(\forall i\in I)\ M_{q_i}\subseteq M_p$\lo. Since $\mc M$ is spectral, we have that $(\forall i\in I)\ q_i\precsim p$. If $I$ is finite, then there exists a finite $C>0$ such that $(\forall i\in I)\ q_i\leq Cp$. Again, by spectrality we conclude that 
$\bigvee_{i\in I} q_i\precsim p$, finishing the proof.
\end{proof}

\begin{prop}
Finite intersections 
preserve both fullness and $A$-convexity.
\end{prop}
\begin{proof}
Fullness is obviously preserved.
Now, let $N_1,N_2\subseteq M$. It is clear that
\[
\mc P(N_1\cap N_2) = \mc P(N_1)\cap\mc P(N_2).
\]
Thus, given $p_1,p_2\in \mc P(N_1\cap N_2)$, if $N_1$ and $N_2$ are $A$-convex then
\[
p_1\vee p_2 \in \mc P(N_1)\cap\mc P(N_2) = \mc P(N_1\cap N_2),
\] 
i.e.\ $N_1\cap N_2$ is $A$-convex.
\end{proof}

\begin{rk}
Up to now, we have only used the first condition in the definition of spectrality. We will make essential use of the second one in the next section.
\end{rk}

\subsection{Schwartz schemes}

Loosely speaking, a \emph{scheme} is a sheaf $\mc A:\cat{Op}(M)\rightarrow \cat{Alg}$ from the category of open subsets of a topological space $M$ (with inclusions as morphisms) to the category of algebras which is such that:
\begin{enumerate} 
	\item $\mc A(U)$ is an algebra of observables over $U$.
	\item Every point $m\in M$ has a neighborhood $U$ that can be uniquely reconstructed (up to homeomorphism) as the spectrum of $\mc A(U)$. 
\end{enumerate}
Thus, to start with, we need a suitable topology on $M$.



\begin{defn}
Given a seminorm $p\in P$,
we say that $p$ is \emph{strictly dominated} by $q\in P$,
written $p\prec q$, whenever $p\precsim q$ and $A_q\rightarrow A_p$ is compact. 
If there exists such a $q$, we say that $p$ is \emph{compact.}
\end{defn}
\begin{rk}
By Schauder's theorem, $A_q\rightarrow A_p$ is compact if, and only if, $A_p^*\rightarrow A_q^*$ is compact, i.e.\ 
\[
\Set{ \norm{\ }_{A_p^*} \leq 1 } = \Set{\norm{\ }_{A_p}\leq 1}^\circ \text{ is precompact in } A_q^*\Loo,
\]
where $(\cdot)^\circ$ denotes the polar of its argument. 
Now, by the Mittag-Leffler theorem, $\set{p\leq 1}$ is dense in $\set{\norm{\ }_{A_p} \leq 1}$, and thus $p\prec q$ if, and only if, $\Set{p\leq 1}^\circ$ is precompact in $A_q^*$\Loo.
\end{rk}

\begin{defn}
A subset $Q\subseteq P$ is said to be \emph{strictly increasing} if 
\[
(\forall p\in Q)(\exists q\in Q)\ p\prec q.
\]
If $P$ itself is strictly increasing, then $A$ is said to be a \emph{Schwartz space.}
\end{defn}

From now on, $\cat{sAlg}$ will be the subcategory of those $A\in\cat{amAlg}$ which are Schwartz spaces and satisfy the condition
\[
(\forall p,q\in P)\ A_p^*\cap A_q^* = A_{p\wedge q}^*\Loo.
\]
So, suppose that $A\in\cat{sAlg}$.
We are in position to introduce the desired topology on $M$. Given a strictly increasing subset $Q\subseteq P$, we let
\[ 
U_Q = \bigcup_{p\in Q} M_p\subseteq M.
\] 

\begin{prop}
The family $\Set{U_Q | Q\subseteq P \text{ strictly increasing}}$  is closed under finite intersections. It is, in particular, a basis of open sets for a topology on $M$.
\end{prop}
\begin{proof}
By spectrality,  $\mc M$ does not play any role here: given a finite family $\{Q_i\}_{i\in I}$ of strictly increasing subsets of $P$, the statement is equivalent to
\[
Q = \Set{ \bigwedge q_i | q_i\in Q_i}
\]
being strongly inceasing. So, let $p_i\lo, q_i\in Q_i$ be such that $q_i\prec p_i$\lo, i.e.\ 
$\Set{q_i\leq 1}^\circ$ is precompact in $A_{p_i}^*$\LOO. Then,
\[
\Set{\bigwedge q_i \leq 1}^\circ \subseteq  \bigcap \Set{q_i\leq 1}^\circ \text{ is precompact in }\bigcap A_{p_i}^* = A_{\bigwedge p_i}^*\LOO,
\]
concluding the proof. 
\end{proof}


We now recall the basic definitions in sheaf theory.
\begin{defn}
Let $\cat{Cat}$ be some category and $X\in\cat{Top}$ a topological space.
A $\cat{Cat}$-valued \emph{presheaf} on  $X$ is a contravariant functor $\mc F: \cat{Op}(X)\rightarrow \cat{Cat}$. 
If $U,V\in\cat{Op}(X)$, $V\subseteq U$ and $f\in\mc F(U)$, then the image of $f$ under the morphism $\mc F(U)\rightarrow \mc F(V)$ is conveniently written $f|_V$\lo.
We will also write, given $U\in\cat{Op}(X)$,
\[
\mc F|_U = \cat{Op}(U)\rightarrow \cat{Op}(X)\xrightarrow{\mc F} \cat{Cat}.
\]
Finally, $\mc F$ is said to be a \emph{sheaf} if it is:
\begin{enumerate}
	\item \emph{Separated:} if $f\in \mc F(\bigcup U_i)$ is such that $f|_{U_i} = 0$, then $f=0$.
	\item \emph{Complete:} if $f_i\in \mc F(U_i)$ are such that $f_i|_{U_i\cap U_j} = f_j|_{U_i\cap U_j}$\Lo, then there exists an $f\in \mc F\left(\bigcup U_i\right)$ such that $f|_{U_i} = f_i$\lo.
\end{enumerate}
\end{defn}

We emphasize that we are assuming that a choice of spectral functor 
\(
\mc M 
\)
has been made. The forthcoming definitions are dependent on this choice. 

\begin{prop}
Given $U\in \cat{Op}(M)$, let
\(
\mc O_A(U) = 
	\lim_{p\in \mc P(U)} A_p\lo.
\)
This defines 
a separated presheaf $\mc O_A: \cat{Op}(M)\rightarrow \cat{sAlg}$. 
\end{prop}
\begin{proof}
It is straightforward that $\mc O_A(U)\in\cat{sAlg}$. 
Functoriality follows from the fact that 
$\mc P$ is a functor from  $\cat{Op}(M)$ to the category $\cat{Dir}(P)$ of directed subsets of $P$. Finally, in order to prove separatedness, let $a\in \mc O_A\bigl(\bigcup_{i=1}^n U_i\bigr)$ with $U_i = \bigcup_{p\in Q_i} M_p$\lo, $Q_i\subseteq P$. Thus,
\[
\mc O_A\bigl(\bigcup U_i\bigr) = \lim_{p\in\bigvee Q_i} A_p\lo.
\]
Suppose that $a|_{U_i} = 0$, i.e.\ $q(a)=0$, for all $q\in Q_i$\lo, and let $p\in\bigvee Q_i$\lo, say $p=q_1\vee\cdots\vee q_n$ with $q_i\in Q_i$\lo. Obviously,
\(
p(a) = \max q_i(a)=0,
\)
which proves that $a=0$.
\end{proof}
\begin{rk}
If $U\in\cat{Op}(M)$ is full and $A$-convex, then $\mc M(\mc O_A(U)) = U$.
\end{rk}
\begin{rk}
One might want to have similar results for subcategories of $\cat{amAlg}$ other than $\cat{sAlg}$. For instance, if $A$ is Fréchet, one would like the algebras $\mc O_A(U)$ to be Fréchet, too. This is ensured by the fact that the separated quotient of Fréchet spaces is Fréchet. The same is true if we replace barreled for Fréchet, see~\cite{m:Scha71} for both assertions.
\end{rk}

We are ready to state the main definition of this section.

\begin{defn}
A \emph{Schwartz scheme} is a sheaf $\mc A:\cat{Op}(M)\rightarrow \cat{sAlg}$ such that every point $m\in M$ has a neighborhood $U$ satisfying 
\[
\mc A|_U \cong \mc O_{\mc A(U)}\Loo.
\]
Note that this means, in particular, that
\(
U \cong \mc M(\mc A(U)).
\)
An \emph{affine Schwartz scheme} is one such that $\mc A=\mc O_A$\lo, with $A\cong\mc A(M)$.
\end{defn}
\begin{rk}
Given  $A\in\cat{sAlg}$,
the presheaf $\mc O_A$ might not be complete. Thus, an important problem would be to find reasonable conditions ensuring that it is. 
\end{rk}
\begin{rk}
As it follows from Theorem~\ref{reconstruction}, there is a bijective correspondence between complex manifolds and Schwartz schemes of algebras satisfying the geometric hypothesis of Subsection~\ref{geometric hypothesis}.
\end{rk}

\appendix

\bibliographystyle{amsplain}
\bibliography{math}

\end{document}